\documentclass[11pt]{amsart}
\usepackage{amsmath,amssymb,amsthm}
\usepackage{color,url}

\makeatletter
\newcommand{\diagdots}[3][-25]{%
  \rotatebox{#1}{\makebox[0pt]{\makebox[#2]{\xleaders\hbox{$\cdot$\hskip#3}\hfill\kern0pt}}}%
}
\newtheorem{theorem}{Theorem}[section]
\newtheorem{lemma}[theorem]{Lemma}
\newtheorem{corollary}[theorem]{Corollary}
\newtheorem{proposition}[theorem]{Proposition}
\newtheorem{remark}[theorem]{Remark}

\theoremstyle{definition}
\newtheorem{observation}[theorem]{Observation}
\newtheorem{de}[theorem]{Definition}

\makeatother

\def\Sym{\mathop{\mathrm{Sym}}\nolimits}

\def\an{\mathop{\mathrm{Ann}}\nolimits}
\def\al{\mathop{\mathrm{Ann}_L}\nolimits}
\def\ar{\mathop{\mathrm{Ann}_R}\nolimits}
\def\anl{\mathop{\mathcal{A}_L}\nolimits}
\def\anr{\mathop{\mathcal{A}_R}\nolimits}
\def\FF{{\mathbb F}}
\def\NN{{\mathbb N}}
\def\ZZ{{\mathbb Z}}
\def\BB{{\mathbb B}}

\def\Aut{\mathrm{Aut}}
\def\M{{\mathrm M}}
\def\N{{\mathrm N}}
\def\Z{{\mathrm Z}}
\def\V{{\mathrm V}}

\DeclareFontFamily{U}{mathx}{\hyphenchar\font45}
\DeclareFontShape{U}{mathx}{m}{n}{
      <5> <6> <7> <8> <9> <10>
      <10.95> <12> <14.4> <17.28> <20.74> <24.88>
      mathx10
      }{}
\DeclareSymbolFont{mathx}{U}{mathx}{m}{n}
\DeclareMathSymbol{\bigtimes}{1}{mathx}{"91}

\begin{document}

\title[Automorphism group of zero-divisor digraph]{The automorphism group of the zero-divisor digraph of matrices over an antiring}
\author[D. Dol\v zan]{David Dol\v zan}
\address{Faculty of Mathematics and Physics, University of Ljubljana, Jadranska 21, 1000 Lju\-blja\-na, Slovenia\newline
\indent IMFM, Jadranska 19, 1000 Ljubljana, Slovenia} \email{david.dolzan@fmf.uni-lj.si}
\author[G. Verret]{Gabriel Verret}
\address{Department of Mathematics, University of Auckland, Private Bag 92019, Auckland 1142, New Zealand} \email{g.verret@auckland.ac.nz}

\date{\today}

\begin{abstract}
We determine the automorphism group of the zero-divisor digraph of the semiring of matrices over an antinegative commutative semiring with a finite number of zero-divisors.
\end{abstract}
\maketitle


\section{Introduction}

In recent years, the zero-divisor graphs of various algebraic structures have received a lot of attention, since they are a useful tool for revealing the algebraic properties through their graph-theoretical properties. In 1988, Beck \cite{beck88} first introduced the concept of the zero-divisor graph of a commutative ring. In 1999, Anderson and Livingston \cite{andliv99} made a slightly different definition of the zero-divisor graph in order to be able to investigate the zero-divisor structure of commutative rings. In 2002, Redmond \cite{redmond02} extended this definition to also include  non-commutative rings. Different authors then further extended this concept to semigroups  \cite{dem02}, nearrings \cite{cann05} and  semirings \cite{atani08}.

Automorphisms of graphs play an important role both in graph theory and in algebra, and  finding the automorphism group of certain graphs is often very difficult. Recently, a lot of effort has been made to determine the automorphism group of various zero-divisor graphs.  In \cite{andliv99}, Anderson and Livingston proved that $\Aut(\Gamma(\ZZ_n))$ is a direct product of symmetric groups for $n \geq 4$ a non-prime integer. In the non-commutative case, the case of matrix rings and semirings is especially interesting. Thus,   it was shown in \cite{han} that, when $p$ is a prime, $\Aut(\Gamma(\M_2(\ZZ_p)))\cong\Sym(p + 1)$. More generally,  it was proved in \cite{park}, that  $\Aut(\Gamma(\M_2(\FF_q)))\cong\Sym(q+1)$. In \cite{wong}, the authors determined the automorphism group of the zero-divisor graph of all rank one upper triangular matrices over a finite field, and in \cite{wang} they determined the automorphism group of the zero-divisor graph of the matrix ring of all upper triangular matrices over a finite field. Recently, the automorphism group of the zero-divisor graph of the complete matrix ring of matrices over a finite field have been found independently in \cite{wang1} and \cite{zhou}.

In this paper, we study the zero-divisor graph of matrices over commutative semirings. The theory of semirings has many applications in optimization theory, automatic control, models of discrete event networks and graph theory (see e.g. \cite{baccelli, cuninghame, li2014, zhao2010}) and the zero-divisor graphs of semirings were recently studied in \cite{atani10, dolzan, patil}.  For an extensive theory of semirings, we refer the reader to \cite{hebisch}. There are many natural examples of commutative semirings, for example, the set of nonnegative integers (or reals) with the usual operations of addition and multiplication. Other examples include distributive lattices, tropical semirings, dio\"{\i}ds, fuzzy algebras, inclines and bottleneck algebras. 

The theory of matrices over semirings differs quite substantially from the one over rings, so the methods we use are necessarily distinct from those used in the ring setting. The main result of this paper is the determination of the automorphism group of the zero-divisor digraph of a semiring of matrices over an  antinegative commutative semiring with a finite number of zero-divisors (see Theorem~\ref{characterization}).

\section{Definitions and preliminaries}
\subsection{Digraphs}
A \emph{digraph} $\Gamma$ consists of a set $\V(\Gamma)$ of \emph{vertices}, together with a binary relation $\rightarrow$ on $\V(\Gamma)$. An \emph{automorphism} of $\Gamma$ is a permutation of $\V(\Gamma)$ that preserves the relation $\rightarrow$.  The automorphisms of $\Gamma$ form its \emph{automorphism group} $\Aut(\Gamma)$. 

Let $\Gamma$ be a digraph and let $v\in \V(\Gamma)$. We write  $\N^-(v) = \{u \in \V(\Gamma)\colon u\rightarrow v\}$ and $\N^+(v) = \{u \in \V(\Gamma)\colon v\rightarrow u\}$. If, for $u,v \in \V(\Gamma)$, we have $\N^-(u)=\N^-(v)$  and $\N^+(u)=\N^+(v)$, then we say $u$ and $v$ are \emph{twin} vertices. The relation $\sim$ on $\V(\Gamma)$, defined by $u \sim v$ if and only if $u$ and $v$ are twin vertices, is clearly an equivalence relation preserved by $\Aut(\Gamma)$. We will denote by $\overline{\Gamma}$ the factor digraph $\Gamma / \sim$. For $v \in \V(\Gamma)$, we shall denote by $\overline{v}$ the image of $v$ in $\overline{\Gamma}$ and, for  $\sigma\in\Aut(\Gamma)$, by $\overline{\sigma}$ the induced automorphism of $\overline{\Gamma}$. An automorphism $\sigma\in\Aut(\Gamma)$ is called \emph{regular} if $\overline{\sigma}$ is trivial.

\subsection{Semirings}
A \emph{semiring} is a set $S$ equipped with binary operations $+$ and $\cdot$ such that $(S,+)$ is a commutative monoid with identity element 0, and $(S,\cdot)$ is a semigroup. Moreover, the operations $+$ and $\cdot$ are connected by distributivity and 0 annihilates $S$. 

A semiring $S$ is \emph{commutative} if $ab=ba$ for all $a,b \in S$, and  \emph{antinegative} if, for all $a, b \in S$, $a+b=0$ implies that $a=0$ or $b=0$. Antinegative semirings are also called \emph{zerosum-free semirings} or \emph{antirings}.  The smallest
nontrivial example of an antiring is the \emph{Boolean antiring} $\BB=\{0,1\}$ with addition and multiplication defined so that 
$1+1=1\cdot 1=1$.

Let $S$ be a semiring. For $x \in S$, we define the \emph{left and right annihilators in $S$} by $\al(x)=\{y \in S\colon yx=0\}$ and $\ar(x)=\{y \in S\colon xy=0\}$.  If $S$ is commutative, we simply write $\an(x)$ for $\al(x)=\ar(x)$. We denote by $\Z(S)$ the set of \emph{zero-divisors} of $S$, that is $\Z(S)=\{x \in S\colon \exists y \in S\setminus\{0\} \text { such that } xy=0 \text { or } yx=0 \}$.  The \emph{zero-divisor digraph} $\Gamma(S)$ of $S$ is the digraph with vertex-set $S$ and $u\rightarrow v$ if and only if $uv = 0$.

It is easy to see that if $n\geq 1$ and $S$ is a semiring, then the set $\M_n(S)$ of $n\times n$ matrices forms a semiring with respect to matrix addition and multiplication.  If $S$ is antinegative, then so is $\M_n(S)$. If $S$ has an identity $1$, let $E_{ij} \in \M_n(S)$ with entry $1$ in position $(i,j)$, and $0$ elsewhere.

\section{The automorphisms of the zero-divisor digraph}

The following fact will be used repeatedly.
\begin{lemma}
\label{annihilators}
Let $S$ be a semiring. If $A, B\in S$ and  $\sigma \in \Aut(\Gamma(S))$, then
$$\sigma(\al(A))=\al(\sigma(A)) \quad\textrm{ and }\quad \sigma(\ar(A))=\ar(\sigma(A)).$$
\end{lemma}
\begin{proof}
We have
\begin{align*}
X \in \sigma(\al(A)) &\Longleftrightarrow \sigma^{-1}(X)\in \al(A)\\
&\Longleftrightarrow \sigma^{-1}(X)A=0\\
&\Longleftrightarrow X\sigma(A)=0\\
&\Longleftrightarrow X \in \al(\sigma(A)).
\end{align*}
The proof of the second part is analogous.
\end{proof}

\begin{lemma}
\label{additive}
Let $S$ be an antiring and let $\Gamma=\Gamma(S)$. If $A,B\in S$ and $\sigma \in \Aut(\Gamma)$, then $\sigma(A+B)$ and $\sigma(A)+\sigma(B)$ are twin vertices and, in particular, $\overline{\sigma(A+B)}=\overline{\sigma(A)+\sigma(B)}$.
\end{lemma}
\begin{proof}
Using antinegativity, we have
\begin{align*}
X \in \al(\sigma(A+B)) &\Longleftrightarrow X\sigma(A+B)=0\\
&\Longleftrightarrow \sigma^{-1}(X)(A+B)=0\\
&\Longleftrightarrow \sigma^{-1}(X)A=\sigma^{-1}(X)B=0\\
&\Longleftrightarrow X\sigma(A)=X\sigma(B)=0\\
&\Longleftrightarrow X(\sigma(A)+\sigma(B))=0\\
&\Longleftrightarrow X \in \al(\sigma(A)+\sigma(B)).
\end{align*}
We have proved that $\al(\sigma(A+B)) = \al(\sigma(A)+\sigma(B))$. An analogous proof yields $\ar(\sigma(A+B)) = \ar(\sigma(A)+\sigma(B))$.
This implies that $\sigma(A+B)$ and $\sigma(A)+\sigma(B)$ are twin vertices.
\end{proof}

\begin{de}
Let $S$ be a commutative semiring, let $n\in\NN$ and let $A \in \M_n(S)$ with $(i,j)$ entry $a_{ij}$. For every $i,j \in\{1,\ldots,n\}$, we define $C_i(A)=\bigcap_{k=1}^n{\an(a_{ki})}$ and $R_j(A)=\bigcap_{k=1}^n{\an(a_{jk})}$. Let  $\anr(A):=(C_1(A),\ldots,C_n(A))  \in S^{n}$ and  $\anl(A):=(R_1(A),\ldots,R_n(A)) \in S^{n}$.
\end{de}

The next theorem characterizes the twin vertices of $\Gamma(\M_n(S))$.

\begin{theorem}
\label{twin}
Let $S$ be a commutative antiring, let $n\in\NN$ and let $A,B\in \M_n(S)$. Then $A$ and $B$ are twin vertices of $\Gamma(\M_n(S))$ if and only if $\anl(A)=\anl(B)$ and $\anr(A)=\anr(B)$.
\end{theorem}
\begin{proof}
Let $a_{ij}$ and $b_{ij}$ be the $(i,j)$ entry of $A$ and $B$, respectively. Suppose first that $A$ and $B$ are twin vertices of $\Gamma(\M_n(S))$ and assume that $\anr(A) \neq \anr(B)$. This implies that, for some $i \in\{1,\ldots,n\}$, we have $C_i(A) \neq C_i(B)$. Swapping the role of $A$ and $B$ if necessary, there exists $s \in S$ such that $s \in C_i(A)$ and $s \notin C_i(B)$. Therefore, there exists $k \in\{1,\ldots,n\}$ such that $s \notin \an(b_{ki})$. Now, let $C=sE_{ik} \in \M_n(S)$ and observe that $AC=0$ but $BC \neq 0$, so $\N^+(A) \neq \N^+(B)$, which is a contradiction with the fact that $A$ and $B$ are twin vertices. We have thus proved that $\anr(A)=\anr(B)$. A similar argument yields that $\anl(A)=\anl(B)$.

Conversely, assume now that $\anl(A)=\anl(B)$ and $\anr(A)=\anr(B)$. Suppose there exists $X \in \M_n(S)$ such that $AX=0$. Therefore, for all $i,j \in\{1,\ldots,n\}$ we have $\sum_{k=1}^n{a_{ik} x_{kj}}=0$. Since $S$ is an antiring, this further implies that $a_{ik} x_{kj}=0$ for all $i, j, k\in\{1,\ldots,n\}$. So, $x_{kj} \in \an(a_{ik})$ and therefore $x_{kj} \in C_k(A)=C_k(B)$ for all $k \in\{1,\ldots,n\}$. Thus, for all  $i, j, k \in\{1,\ldots,n\}$, we have $x_{kj} \in \an(b_{ik})$. This yields
$b_{ik} x_{kj} = 0$ for all $i, j, k\in\{1,\ldots,n\}$, so $BX=0$. Thus, we have proved that $\N^+(A) \subseteq \N^+(B)$. By swapping the roles of $A$ in $B$ we also get 
$\N^+(B) \subseteq \N^+(A)$, so $\N^+(A) = \N^+(B)$. A similar argument yields that $\N^-(A)=\N^-(B)$, thus $A$ and $B$ are twin vertices.
\end{proof}

%
%
%
%
%

\begin{de}
Let $S$ be a commutative semiring and let $\alpha\in S\setminus \Z(S)$. We say that $\alpha = e_1+ e_2+\cdots+e_s$ such that $e_i \neq 0$ for all $i$ and $e_ie_j=0$ for all $i \neq j$ is \emph{a decomposition of $\alpha$ of length $s$}. The \emph{length} $\ell(\alpha)$ of $\alpha$ is the supremum of the length of a decomposition of $\alpha$. We say that $\alpha$ is of \emph{maximal length} if $\ell(\alpha) \geq \ell(\beta)$ for all $\beta \in S \setminus \Z(S)$.

A semiring is \emph{decomposable} if it contains an element of length at least $2$, otherwise it is \emph{indecomposable}.
\end{de}

The next lemma shows that in the case $S$ is decomposable, we can study the automorphisms of the zero-divisor digraph of the matrix ring componentwise.

\begin{lemma}
\label{componentwise}
Let $S$ be a commutative antiring with identity
and let $\alpha \in S \setminus \Z(S)$ be of maximal length $s$ with decomposition $\alpha = e_1+ e_2+\cdots+e_s$. Let $n \in \NN$ and $\sigma \in \Aut(\Gamma(\M_n(S)))$. Then  there exists  $\omega\in\Sym(s)$ such that, for every $r\in\{1,\ldots,s\}$, we have $\sigma(e_r \M_n(S)) = e_{\omega(r)} \M_n(S)$.
\end{lemma}
\begin{proof}
Let  $r \in\{1,\ldots,s\}$, let $i,j\in\{1,\ldots,n\}$ and let $B=\sigma(e_rE_{ij})$. So, $\overline{\sigma(e_r E_{ij})}$ 
$=\overline{\sum_{k=1}^s e_{k}B}$. 
By Lemma~\ref{additive}, we have $\overline{e_rE_{ij}} = \overline{\sum_{k=1}^s\sigma^{-1}(e_kB)}$. Since $S$ is antinegative, for every $k \in\{1,\ldots,s\}$, there exists $f_{k} \in S$ such that $\sigma^{-1}(e_kB)=f_{k}E_{ij}$ and $\sum_{k=1}^sf_{k}=e_rz$ for some $z \in S\setminus\Z(S)$.

Let $k,k'\in\{1,\ldots,s\}$. We have $\al(f_{k}E_{ij}), \al(f_{k'}E_{ij}) \subseteq \al(f_{k}f_{k'}E_{ij})$ hence $\al(e_{k}B), \al(e_{k'}B) \subseteq \al(\sigma(f_{k}f_{k'}E_{ij}))$. If $k \neq k'$, then $ \M_n(\alpha S)\subseteq \al(e_{k}B) + \al(e_{k'}B)$, which is possible only if  $f_{k}f_{k'}=0$. We have shown that  $f_{k}f_{k'}=0$ for every $k,k'\in\{1,\ldots,s\}$ with $k\neq k'$.

Since $\alpha$ is not a zero-divisor, we have $\overline{\Gamma(S)}=\overline{\Gamma(\alpha S)}$  and, by Theorem~\ref{twin}, also $\overline{\Gamma(\M_n(S))}=\overline{\Gamma(\M_n(\alpha S))}$. For $t\in\{1,\ldots,s\}$, $t \neq r$, we have $\sum_{k=1}^sf_{k}e_t=e_re_tz=0$. By antinegativity, this implies  that   $f_ke_t=0$ for every $k\in\{1,\ldots,s\}$. It follows that 
$$\alpha z=\sum_{i\neq r} e_i z+\sum_{k=1}^s f_k$$
is a decomposition of $\alpha z$. Since  $z\notin\Z(S)$, $e_iz\neq 0$ and, since $\ell(z \alpha) \leq \ell (\alpha)=s$, it follows that all but exactly one of the $f_k$'s are $0$. This implies that all but one of the $e_kB$'s  are $0$ and there exists $k\in\{1, \ldots, s \}$ such that  $\overline{\sigma(e_r E_{ij})} = \overline{e_k B}$. This shows the existence of a permutation $\omega\in\Sym(s)$ such that $\overline{\sigma(e_r E_{ij})} = \overline{e_{\omega(r)} B}$.

Let $t\in\{1,\ldots,n\}$. Since $\alpha\notin\Z(S)$, we have $\alpha e_r \neq 0$ thus $e_r^2 \neq 0$,  $e_r E_{ij} e_r E_{j t} \neq 0$ and  $\overline{(e_{\omega(r)}B) \sigma(e_r E_{jt})}\neq 0$. Since $\overline{\Gamma(\M_n(S))}=\overline{\Gamma(\M_n(\alpha S))}$, this implies $\overline{\sigma(e_r E_{jt})} \in \overline{e_{\omega(r)} \M_n(S)}$. 

As this holds for all $j,t\in\{1,\ldots,n\}$, we have $\overline{\sigma(e_r \M_n(S))} \subseteq\overline{e_{\omega(r)} \M_n(S)}$.  However, a twin vertex to a vertex from from $e_r \M_n(S)$ is itself in $e_r \M_n(S)$, therefore also $\sigma(e_r \M_n(S)) \subseteq e_{\omega(r)} \M_n(S)$. Since  $\sigma$ is a bijection, $\sigma(e_r \M_n(S)) = e_{\omega(r)} \M_n(S)$. 

%
%
%
\end{proof}

\begin{lemma}
\label{max}
Let $S$ be a commutative antiring and let $\alpha \in S \setminus \Z(S)$ be of maximal length $s$ with  decomposition $\alpha = e_1+ e_2+\cdots+e_s$. Then, for every $i\in\{1,\ldots,s\}$, the subsemiring $e_iS$ is indecomposable.
\end{lemma}
\begin{proof}
Suppose that $e_iS$ is decomposable for some $i\in\{1,\ldots,s\}$, say $i=1$ without loss of generality. By definition, there exists $e_1w \in e_1S \setminus \Z(e_1S)$ such that $e_1w = f_1 + f_2$, where $f_1, f_2 \in e_1S\setminus\{0\}$ and $f_1f_2=0$. For all $j \neq 1$, we have $e_je_1w=0$ and thus $e_jf_1 = e_jf_2=0$ by antinegativity. Let  $\beta=e_1w + e_2+\cdots+e_s$. 

Suppose that $\beta x=0$ for some $x \in S$. By antinegativity, we have $(e_1w)(e_1x) = 0$ and $e_2x=\cdots=e_sx=0$. Since $e_1w$ is not a zero-divisor in $e_1S$ this implies that $e_1x=0$ and therefore also $\alpha x = 0$.  However, $\alpha$ is not a zero-divisor, so we can conclude that $x=0$.

This shows that $\beta=f_1+f_2 + e_2+\cdots+e_s$ is not a zero-divisor in $S$, which is a contradiction with the maximal length of $\alpha$.
\end{proof}

We first focus on the automorphisms restricted to the matrices over indecomposable subsemirings.

\begin{proposition}
\label{elementary}
Let $S$ be a commutative antiring with identity 
and let $\alpha \in S \setminus \Z(S)$ be of maximal length $s$ with  decomposition $\alpha = e_1+ e_2+\cdots+e_s$. Let $u, v \in \{1, \ldots, s \}$, $S_1=e_uS$ and $S_2=e_vS$. Let $n \in \NN$ and $\sigma \in \Aut(\Gamma(\M_n(S)))$ such that $\sigma(\M_n(S_1))=\M_n(S_2)$.  If $i,j\in \{1, \ldots, n \}$, then there exist $y\in S_2\setminus\Z(S_2)$ and $k,\ell\in \{1, \ldots, n \}$ such that $\sigma(e_uE_{ij}) = yE_{k\ell}$. 
\end{proposition}
\begin{proof}
Write $\sigma(e_uE_{ij}) = \sum_{k,\ell}\beta_{k\ell}E_{k\ell}$. Let $A_{k\ell}=\sigma^{-1}(\beta_{k\ell}E_{k\ell})$. By Lemma~\ref{additive}, $\sigma(e_uE_{ij})$ and $\sigma\left(\sum_{k,\ell}A_{k\ell}\right)$ are twin vertices, therefore $e_uE_{ij}$ and $\sum_{k,\ell}A_{k\ell}$ are twin vertices as well. Now, twin vertices of $e_uE_{ij}$ must be of the form $zE_{ij}$, so $\sum_{k,\ell}A_{k\ell}=zE_{ij}$ for some $z \in S_1$.
Since $S$ is antinegative, we can conclude that, for all $k,\ell\in \{1, \ldots, n \}$, there exist
$\alpha_{k\ell} \in S_1$ such that $A_{k\ell}=\alpha_{k\ell}E_{ij}$ and $\sum_{k,\ell}\alpha_{k\ell}=z$. 

Let $k,k',\ell,\ell'\in \{1, \ldots, n \}$ with $(k,\ell) \neq (k',\ell')$. Now, we either have $k \neq k'$ or $\ell \neq \ell'$. Suppose first that $k \neq k'$. Since $S$ is commutative, we have $\al(A_{k\ell})=\al(\alpha_{k\ell}E_{ij})\subseteq \al(\alpha_{k\ell}\alpha_{k'\ell'}E_{ij})$. By Lemma~\ref{annihilators}, this implies $\al(\beta_{k\ell}E_{k\ell}) \subseteq \al(\sigma(\alpha_{k\ell}\alpha_{k'\ell'}E_{ij}))$. Similarly, we have $\al(\beta_{k'\ell'}E_{k'\ell'}) \subseteq \al(\sigma(\alpha_{k\ell}\alpha_{k'\ell'}E_{ij}))$. Since $k \neq k'$, $\M_n(S)=\al(\beta_{k\ell}E_{k\ell}) + \al(\beta_{k'\ell'}E_{k'\ell'}) \subseteq \al(\sigma(\alpha_{k\ell}\alpha_{k'\ell'}E_{ij}))$, which implies  $\sigma(\alpha_{k\ell}\alpha_{k'\ell'}E_{ij})=0$ and thus $\alpha_{k\ell}\alpha_{k'\ell'}=0$. If $\ell \neq \ell'$, we arrive at the same conclusion by using right annihilators, namely that distinct $\alpha_{k\ell}$'s annihilate each other. Since $S_1$ is indecomposable by Lemma \ref{max}, the sum $\sum_{k,\ell}\alpha_{k\ell}=z$ has at most one non-zero summand. It follows that there is at most one non-zero $A_{k\ell}$ and at most one non-zero $\beta_{k\ell}E_{k\ell}$. This concludes the proof of the first part, with $y=\beta_{k\ell}$.

It remains to show that $y \notin \Z(S_2)$. Suppose, on the contrary, that $y \in \Z(S_2)$. By the first part of the result, there exist $y'\in S_1$ and $i',j'\in \{1, \ldots, n \}$ such that $\sigma^{-1}(e_vE_{k\ell})=y'E_{i'j'}$. Since $y \in \Z(S_2)$, we have $\al(e_vE_{k\ell}) \subset \al(yE_{k\ell})$ and $\ar(e_vE_{k\ell}) \subset \ar(yE_{k\ell})$. By Lemma~\ref{annihilators}, it follows  that $\al(y'E_{i'j'}) \subset \al(e_uE_{ij})$ and of course also $\ar(y'E_{i'j'}) \subset \ar(e_uE_{ij})$. This is only possible if $i=i'$ and $j=j'$ which implies $\al(y'E_{ij}) \subset \al(e_uE_{ij})$, a contradiction.
\end{proof}

%
%
%

\begin{lemma}
\label{divisors}
Let $S$ be a commutative antiring with identity 
and let  $\alpha \in S \setminus \Z(S)$ be of maximal length $s$ with  decomposition $\alpha = e_1+ e_2+\cdots+e_s$. Let $u, v \in \{1, \ldots, s \}$, $S_1=e_uS$ and $S_2=e_vS$. Let $n \in \NN$ and $\sigma \in \Aut(\Gamma(\M_n(S)))$ such that $\sigma(\M_n(S_1))=\M_n(S_2)$.  If $x \in \Z(S_1)$ and $i,j\in \{1, \ldots, n \}$, then there exist $z\in \Z(S_2)$ and $k,\ell\in \{1, \ldots, n \}$ such that $\sigma(xE_{ij}) = zE_{k\ell}$.
\end{lemma}
\begin{proof}
By Proposition~\ref{elementary}, we know that $\sigma(e_uE_{ij}) = y E_{k\ell}$ for some $y \notin \Z(S_2)$ and $k,\ell\in \{1, \ldots, n \}$.  Since $x \in \Z(S_1)$, we have $\al (E_{ij}) \subset \al(xE_{ij})$ and $\ar (E_{ij}) \subset \ar(xE_{ij})$. By Lemma~\ref{annihilators}, it follows that $\al (yE_{k\ell}) \subset \al(\sigma(xE_{ij}))$ and also $\ar (yE_{k\ell}) \subset \ar(\sigma(xE_{ij}))$. This implies that all entries of  $\sigma(xE_{ij})$ are zeros except entry $(k,\ell)$, so $\sigma(xE_{ij})=zE_{k\ell}$ for some $z \in S_2$. Because $\al (yE_{k\ell}) \neq \al(\sigma(xE_{ij}))=\al(zE_{k\ell})$ and $\ar (yE_{k\ell}) \neq \ar(\sigma(xE_{ij}))=\ar(zE_{k\ell})$, we must have $z\in \Z(S_2)$.
\end{proof}

\begin{lemma}
\label{permutations}
Let $S$ be a commutative antiring with identity 
and let  $\alpha \in S \setminus \Z(S)$ be of maximal length $s$ with  decomposition $\alpha = e_1+ e_2+\cdots+e_s$. Let $u, v \in \{1, \ldots, s \}$, $S_1=e_uS$ and $S_2=e_vS$. Let $n \in \NN$ and $\sigma \in \Aut(\Gamma(\M_n(S)))$ such that $\sigma(\M_n(S_1))=\M_n(S_2)$.  Then there exists $\pi \in \Sym(n)$ such that 
$\overline{\sigma(e_uE_{ij})} = \overline{e_vE_{\pi(i)\pi(j)}}$ for all $i,j\in\{1,\ldots,n\}$. 
\end{lemma}
\begin{proof}
Let $i,j,j'\in\{1,\ldots,n\}$ with $j\neq j'$. By Proposition~\ref{elementary}, there exist $k,k',\ell,\ell'\in\{1,\ldots,n\}$ such that  $\overline{\sigma(e_uE_{ij})} = \overline{e_vE_{k\ell}}$ and $\overline{\sigma(e_uE_{ij'})} = \overline{e_vE_{k'\ell'}}$. For all $r,s\in\{1,\ldots,n\}$ with $s \neq i$, we have $e_uE_{rs}(e_uE_{ij}+e_uE_{ij'})=0$. By Lemma~\ref{additive}, this implies that $\overline{\sigma(e_uE_{rs})(e_uE_{k\ell}+e_uE_{k'\ell'})}=0$ and thus $\overline{\left(\sum_{r,s \neq i}\sigma(e_uE_{rs})\right)(e_uE_{k\ell}+e_uE_{k'\ell'})}=0$. By Proposition~\ref{elementary}, $\overline{\sigma(e_uE_{rs})}=\overline{e_vE_{r's'}}$ for some $r',s' \in\{1,\ldots,n\}$. Since $\sigma$ is a permutation, $\sum_{r,s \neq i}\sigma(e_uE_{rs})$ is a matrix with exactly $n$ entries equal to $0$. It follows that $k=k'$. 

By the paragraph above, there exists $\pi \in \Sym(n)$ such that $\overline{\sigma(e_uE_{ab})}=\overline{e_vE_{\pi(a)c}}$, for some $c$. A similar argument yields that there exists a permutation  such that  $\overline{\sigma(e_uE_{ab})}=\overline{e_vE_{c\pi'(b)}}$, for some $c$. However, for every $j,k\in\{1,\ldots,n\}$ with $j\neq k$, we have $E_{jj}E_{kk}=0$ and thus $E_{\pi(j)\pi'(j)}E_{\pi(k)\pi'(k)}=0$. This implies that $\pi(k) \neq \pi'(j)$ for every $k \neq j$, so $\pi(j)=\pi'(j)$.
Therefore $\pi'=\pi$.
\end{proof}



For $\pi\in\Sym(n)$ and $A\in\M_n(S)$, let $\theta_\pi(A)$ be the matrix obtained from $A$ by applying the permutation $\pi$ to its rows and columns. Note that  $\theta_\pi$ induces a permutation of  $\M_n(S)$. 

\begin{corollary}
\label{inner}
Let $S$ be a commutative antiring with identity 
and let  $\alpha \in S \setminus \Z(S)$ be of maximal length $s$ with  decomposition $\alpha = e_1+ e_2+\cdots+e_s$. Let $u, v \in \{1, \ldots, s \}$, $S_1=e_uS$ and $S_2=e_vS$. Let $n \in \NN$ and $\sigma \in \Aut(\Gamma(\M_n(S)))$ such that $\sigma(\M_n(S_1))=\M_n(S_2)$. Then there exist $\pi\in\Sym(n)$ and $\tau$ an isomorphism from $\Gamma(S_1)$ to $\Gamma(S_2)$ such that, if we extend $\tau$ to a mapping $\M_n(S_1)\to\M_n(S_2)$ and restrict $\sigma$ to $\M_n(S_1)$, then $\overline{\sigma}=\overline{\theta_\pi\circ \tau}$.
\end{corollary}
\begin{proof}
By Lemma~\ref{permutations}, there exists $\pi \in \Sym(n)$ such that $\overline{\sigma(e_uE_{ij})} = \overline{\theta_\pi(e_vE_{ij})}$ for all $i,j\in\{1,\ldots,n\}$. Let $\rho= \theta_\pi^{-1}\circ \sigma$ and note that  $\rho\in\Aut(\Gamma(\M_n(S)))$ and we have $\overline{\rho(e_uE_{ij})}=\overline{e_vE_{ij}}$ for all $i, j \in\{1,\ldots,n\}$. 

Let $x \in \Z(S_1)$ and $i, j,j' \in\{1,\ldots,n\}$. Clearly, $\rho(\M_n(S_1))=\M_n(S_2)$ so, by Lemma~\ref{divisors}, there exist $z,z' \in \Z(S_2)$ such that  $\rho(xE_{ij})=zE_{ij}$  and $\rho(xE_{ij'})=z'E_{ij'}$. 

We show that $\an(z) \subseteq \an(z')$. Let $a\in S_2$ such that $az=0$. Note that $(aE_{ii})(zE_{ij})=0$. Since $a\in\Z(S_2)$, Lemma~\ref{divisors} implies that there exists $b \in \Z(S_1)$ such that $aE_{ii}=\rho(b E_{ii})$, hence $\rho(b E_{ii}) \rho(x E_{ij}) = 0$ and therefore also $(b E_{ii})(xE_{ij}) = 0$ which implies $bx = 0$. It follows that   $(b E_{ii}) (x E_{ij'})=0$, $\rho(b E_{ii}) \rho(x E_{ij'}) = 0$ and $(aE_{ii})(z'E_{ij'})=0$ which yields $az'=0$. 

We have shown that $\an(z) \subseteq \an(z')$. A symmetrical argument yields $\an(z') \subseteq \an(z)$ hence $\an(z)=\an(z')$ which implies that $\overline{\rho(xE_{ij'})}=\overline{z'E_{ij'}}=\overline{zE_{ij'}}$. A similar argument shows that $\overline{\rho(xE_{i'j})}=\overline{zE_{i'j}}$ for all $i' \in\{1,\ldots,n\}$. This implies that $\overline{\rho(xE_{k\ell})}=\overline{zE_{k\ell}}$ for all $k,\ell \in\{1,\ldots,n\}$.

Let $\tau$ denote the mapping $S_1 \rightarrow S_2$ that satisfies $\rho(xE_{11})=\tau(x)E_{11}$. Since $\rho$ is a bijection from $\M_n(S_1)$ to $\M_n(S_2)$, $\tau$ is a bijection from $S_1$ to $S_2$. If $x,y\in S_1$, then $xy=0$ if and only if $(xE_{11})(yE_{11})=0$ if and only if $\tau(x)\tau(y)=0$,  therefore $\tau$ is an isomorphism from $\Gamma(S_1)$ to $\Gamma(S_2)$. Now, extend $\tau$ to an entry-wise mapping $\M_n(S_1)\to \M_n(S_2)$. It is easy to check that $\tau$ induces an isomorphism from $\Gamma(\M_n(S_1))$ to $\Gamma(\M_n(S_2))$ and that, restricted to $\V(\Gamma(\M_n(S_1)))$, we have $\overline{\rho}=\overline{\tau}$. As $\sigma=\theta_\pi \circ \rho$, this concludes the proof.
\end{proof}

We can now join these findings into the following theorem.

\begin{theorem}
\label{characterization}
Let $S$ be a commutative antiring with identity 
and let $\alpha \in S \setminus \Z(S)$ be of maximal length $s$ with decomposition $\alpha = e_1+ e_2+\cdots+e_s$. Let $n \in \NN$ and $\sigma \in \Aut(\Gamma(\M_n(S)))$. Then there exist $\omega \in \Sym(s)$ and, for every $i\in\{1,\ldots,s\}$, there exist $\pi_i\in\Sym(n)$ and an isomorphism $\tau_i : \Gamma(e_iS) \rightarrow \Gamma(e_{\omega(i)}S)$ such that, if we extend $\tau_i$ to a mapping $\M_n(e_iS)\to\M_n(e_{\omega(i)}S)$, then
$$\overline{\sigma(A)}=\overline{\left(\sum_{i=1}^s{(\theta_{\pi_i}\circ\tau_i)(e_iA)}\right)} \text { for all } A \in \M_n(S).$$  

Conversely, if $\omega\in \Sym(s)$ has the property that, for every $i\in\{1,\ldots,s\}$, we have $\Gamma(e_iS)\cong\Gamma(e_{\omega(i)}S)$, $\tau_i$ is an isomorphism from $\Gamma(e_iS)$ to $\Gamma(e_{\omega(i)}S)$ and $\pi_i\in\Sym(n)$, then $\sigma$ defined with $\sigma(A)=\sum_{i=1}^s{(\theta_{\pi_i}\circ\tau_i)(e_iA)}$ is an automorphism of $\Gamma(\M_n(S))$.
\end{theorem}
\begin{proof}
By Lemma~\ref{componentwise}, there exists  $\omega\in\Sym(s)$ such that, for every $i\in\{1,\ldots,s\}$, we have $\sigma(e_i \M_n(S)) = e_{\omega(i)} \M_n(S)$.

By Corollary~\ref{inner}, there exist $\pi_i\in\Sym(n)$ and $\tau_i$ an isomorphism from $\Gamma(e_iS)$ to $\Gamma(e_{\omega(i)}S)$ such that, if we extend $\tau_i$ to a mapping $\M_n(e_iS)\to\M_n(e_{\omega(i)}S)$ and restrict $\sigma$ to $\M_n(e_iS)$, then $\overline{\sigma}=\overline{\theta_{\pi_i}\circ \tau_i}$.

Now, let $A\in \M_n(S)$. By Theorem~\ref{twin},  $\overline{A}=\overline{\alpha A}=\overline{e_1A+ e_2A+\cdots+e_sA}$ and the result follows by Lemma~\ref{additive}.
\end{proof}

The following is a well-known easy exercise.
\begin{observation}\label{obs}
Let $\Gamma$ be a digraph and let $\overline{\Gamma_{t}}$ be the vertex-labelled digraph obtained from $\overline{\Gamma}$ by labelling every $\overline{v}\in\overline{\Gamma}$ with the size of the $\sim$-equivalence class of $v$. Let $\Aut(\overline{\Gamma_{t}})$ be the labelling-preserving group of automorphisms of $\overline{\Gamma_{t}}$ and let $K$ be the group of regular automorphisms of $\Gamma$. If the sizes of the $\sim$-equivalence classes in $\Gamma$ are $c_1,\ldots,c_n$, then $K\cong \prod_{i=1}^n \Sym(c_i)$ and $\Aut(\Gamma)\cong K\rtimes \Aut(\overline{\Gamma_{t}})$.
\end{observation}

\begin{corollary}\label{finalcor}
Let $S$ be a commutative antiring with identity 
and let $\alpha \in S \setminus \Z(S)$ be of maximal length $s$ with decomposition $\alpha = e_1+ e_2+\cdots+e_s$. Say that $e_i$ is equivalent to $e_j$ if $\Gamma(e_iS)\cong\Gamma(e_jS)$. This defines a partition of $\{e_1,\ldots,e_s\}$. Up to relabelling, we may assume that $\{e_1,\ldots,e_p\}$ forms a complete set of representative of the equivalence classes. For $i\in\{1,\ldots,p\}$, let $x_i$ be the size of the equivalence class of $e_i$. Let $n\in\NN$ and let  $K$ be the group of regular automorphisms of $\Gamma(\M_n(S))$. Then  
$$\Aut(\Gamma(\M_n(S)))\cong K\rtimes\left(\prod_{i=1}^p((\Sym(n)\times \Aut(\overline{\Gamma_t}(e_iS)))\wr \Sym(x_i))\right).$$
\end{corollary}


\begin{remark}
By Observation~\ref{obs}, $K\cong \prod_{i=1}^n \Sym(c_i)$, where the $c_i$'s are the sizes of the $\sim$-equivalence classes in $\Gamma(\M_n(S))$. Finding these sizes is in general quite difficult. For example, consider the following very basic situation: let $J\in \M_n(\BB)$ be the all-$1$'s matrix. By Theorem~\ref{twin}, twins of $J$ in $\Gamma(\M_n(\BB))$ are precisely the $n\times n$ $\{0,1\}$-matrices with no row or column of $0$'s. There is no known closed formula for the number of such matrices (see ~\cite{OEIS}).

\end{remark}

\begin{remark}
Throughout the paper, we restricted ourselves to studying  semirings with the property that no non-zero-divisor element can be written as a sum of infinitely many mutually orthogonal zero-divisors.  Obviously, any semiring with a finite set of zero-divisors satisfies this condition.
\end{remark}

\section*{Acknowledgements}
	The first author acknowledges the financial support from the Slovenian Research Agency  (research core funding no. P1-0222).

\end{document}